\documentclass[12pt,a4paper]{article}
\usepackage[utf8]{inputenc}
\usepackage{amsmath}
\usepackage{amsfonts}
\usepackage{indentfirst}
\usepackage{enumerate}
\usepackage{enumitem}
\usepackage{amssymb}
\usepackage{graphicx}
\usepackage{amsthm}
\usepackage{hyphenat}
\usepackage{verbatim}
\usepackage{bold-extra}
\usepackage{changepage}

\author{Luke Warren}
\title{On the iteration of quasimeromorphic mappings}

\date{}

\begin{document}

\maketitle
\begin{center}
School of Mathematical Sciences,
University of Nottingham,

University Park,
Nottingham,
NG7 2RD, UK

Email: \texttt{luke.warren@nottingham.ac.uk}
\end{center}

\newtheorem{thm}{Theorem}[section]
\newtheorem{lem}[thm]{Lemma}
\newtheorem{cor}[thm]{Corollary}
\newtheorem{defn}[thm]{Definition}
\newtheorem{conj}[thm]{Conjecture}
\newtheorem{prop}[thm]{Proposition}
\newtheorem{thm*}{Theorem}
\newtheorem{lem*}{Lemma}
\newtheorem{pf*}{Proof of Theorem}
\newtheorem{conj*}{Conjecture}
\newtheorem{prop*}{Proposition}
\numberwithin{equation}{section}

\newcommand{\capacity}{\operatorname{cap}}
\newcommand{\card}{\operatorname{card}}
\renewcommand{\qedsymbol}{}

	\begin{adjustwidth}{1cm}{1cm}
		\textbf{Abstract.} The Fatou-Julia theory for rational functions has been extended both to transcendental meromorphic functions and more recently to several different types of quasiregular mappings in higher dimensions. We extend the iterative theory to quasimeromorphic mappings with an essential singularity at infinity and at least one pole, constructing the Julia set for these maps. We show that this Julia set shares many properties with those for transcendental meromorphic functions and for quasiregular mappings of punctured space.
	\end{adjustwidth}

\section{Introduction}

The Fatou-Julia theory of analytic and meromorphic functions on the complex plane has been extensively studied by various authors; we refer to \cite{BKL1, Beardon1, Bergweiler1, Dominguez, Milnor} for further information. For these functions, the Fatou set is defined using normal families, while the Julia set is defined as the complement of the Fatou set.

Quasiregular mappings and quasimeromorphic mappings are higher dimensional analogues of analytic and meromorphic functions respectively. In general, due to the dilatation growth of iterates of these functions, following the original Fatou and Julia set definitions does not yield useful results (for example, see \cite[Section~5]{Bergweiler3}). Recently, focus has turned towards using a direct definition of the Julia set for quasiregular mappings, after Sun and Yang \cite{SY1} successfully used the `blowing-up' property of the classical Julia set as the definition itself. Since then, the Julia set of other types of quasiregular mappings have been established and studied using a similar method. Examples of these include quasiregular self-maps of $\hat{\mathbb{R}}^d = \mathbb{R}^d \cup \{\infty\}$ in \cite{Bergweiler4}, entire quasiregular mappings of transcendental type in \cite{BN1}, and quasiregular mappings of punctured space in \cite{NS1}. Here we say that a quasiregular or quasimeromorphic mapping $f$ is of transcendental type if $\lim_{|x| \to \infty} f(x)$ does not exist. The definition of a quasiregular and quasimeromorphic mapping will be given in Section~2.1.

If $f$ is a quasimeromorphic mapping of transcendental type with no poles, then it is precisely an entire quasiregular mapping of transcendental type whose Fatou-Julia theory is covered in \cite{BN1}. As their study requires a different method, throughout this paper we shall only consider quasimeromorphic mappings of transcendental type with at least one pole. Examples of such quasimeromorphic mappings of transcendental type can be found in \cite{FN2}, where the dynamics of a higher dimensional analogue of the meromorphic tangent family were considered.

The aim of this paper is to extend the Fatou-Julia theory to the case of quasimeromorphic mappings of transcendental type with at least one pole. We aim to establish a Julia set definition for these mappings and prove that similar classical properties of Julia sets still hold. These results are stronger than those known for the case of entire quasiregular mappings of transcendental type, since in that case it currently remains conjectured that the Julia set can be written using a cardinality condition as in our definition below.
	
	Suppose that $f: \mathbb{R}^d \to \hat{\mathbb{R}}^d$ is a quasimeromorphic mapping of transcendental type. Then for $x \in \hat{\mathbb{R}}^d$, we define the backward orbit of $x$ as
	\begin{equation*}
	\mathcal{O}^{-}_{f}(x) := \bigcup_{m=0}^{\infty} f^{-m}(x).
	\end{equation*}
	For a set $U \subset \mathbb{R}^d \setminus \mathcal{O}^{-}_{f}(\infty)$, we similarly define the forward orbit of $U$ as
	\begin{equation*}
	\mathcal{O}^{+}_{f}(U) := \bigcup_{m=0}^{\infty} f^{m}(U).
	\end{equation*}
	Finally, we define the exceptional set $E(f)$ as the set of points with a finite backward orbit. As a note, if $f$ is a quasimeromorphic mapping of transcendental type then, by Theorem~\ref{BigPicard} for example, $E(f)$ is finite. Now using this notation, we are considering quasimeromorphic mappings of transcendental type such that $\card(\mathcal{O}^{-}_{f}(\infty)) \geq 2$. 
	
	In keeping with the structure of the Julia set definitions for quasiregular mappings, we shall define the Julia set for quasimeromorphic mappings of transcendental type with at least one pole as follows.
	
\begin{defn}
Let $f: \mathbb{R}^d \to \hat{\mathbb{R}}^d$ be a quasimeromorphic mapping of transcendental type with at least one pole. Then the Julia set $J(f)$ is defined as

\begin{align} \label{JuliaSetDefn}
J(f):= \{ &x \in \hat{\mathbb{R}}^d \setminus \overline{\mathcal{O}^{-}_{f}(\infty)} : \card(\hat{\mathbb{R}}^d \setminus \mathcal{O}^{+}_{f}(U_{x}))< \infty \text{ for all} \nonumber \\ 
&\text{neighbourhoods } U_{x} \subset \hat{\mathbb{R}}^d \setminus \overline{\mathcal{O}^{-}_{f}(\infty)} \text{ of } x \} \cup \overline{\mathcal{O}^{-}_{f}(\infty)}.
\end{align}
\end{defn}
		
	As an immediate remark it can be seen that when $\card(\mathcal{O}^{-}_{f}(\infty))= \infty$, then $J(f) = \overline{\mathcal{O}^{-}_{f}(\infty)}$. This is because for any $U \subset \hat{\mathbb{R}}^d \setminus \overline{\mathcal{O}^{-}_{f}(\infty)}$, we must have that $\mathcal{O}^{+}_{f}(U)$ is disjoint from $\mathcal{O}^{-}_{f}(\infty)$. Hence $\hat{\mathbb{R}}^d \setminus \mathcal{O}^{+}_{f}(U)$ is always infinite.
	
	We show that this definition of the Julia set agrees with the classical definition given for transcendental meromorphic functions with at least one pole. There, the Julia set is defined as the complement of the Fatou set, which is the set of all points that have a neighbourhood in which all iterates are well defined and form a normal family.
	
	\begin{thm}\label{JuliaSetsAgree}
	Let $f : \mathbb{C} \to \hat{\mathbb{C}}$ be a transcendental meromorphic function with at least one pole. Then the classical definition of $J(f)$ agrees with \eqref{JuliaSetDefn}.
	\end{thm}
	
	The main results of this paper are concerned with the properties of $J(f)$ for a quasimeromorphic mapping of transcendental type with at least one pole.
	
	\begin{thm}\label{JuliaSetProperties}
Let $f:\mathbb{R}^d \to \hat{\mathbb{R}}^d$ be a quasimeromorphic mapping of transcendental type with at least one pole. Then the following hold.

\begin{description}
\item[(i)] $J(f) \neq \varnothing$. In fact, $\card(J(f)) = \infty$.
\item[(ii)] $J(f)$ is perfect.
\item[(iii)] $x \in J(f) \setminus \{ \infty \}$ if and only if $f(x) \in J(f)$. In particular, $J(f) \setminus \mathcal{O}^{-}_{f}(\infty)$ is completely invariant.
\item[(iv)] $J(f) \subset \overline{\mathcal{O}^{-}_{f}(x)}$ for every $x \in \hat{\mathbb{R}}^d \setminus E(f)$.
\item[(v)] $J(f) = \overline{\mathcal{O}^{-}_{f}(x)}$ for every $x \in J(f) \setminus E(f)$.
\item[(vi)] Let $U \subset \hat{\mathbb{R}}^d$ be an open set such that $U \cap J(f) \neq \varnothing$. Then for all $x \in \hat{\mathbb{R}}^d \setminus E(f)$, there exists some $w \in U$ and some $k \in \mathbb{N}$ such that $f^k(w)=x$.
\end{description}
\end{thm}

We should remark here that \textbf{(vi)} is a slightly different version of the `blowing-up' property exhibited by (classical) Julia sets. This is to take into account the fact that the forward orbit is not well defined for elements in $\mathcal{O}^{-}_{f}(\infty)$.

A useful property of classical Julia sets is that the Julia set of a function is equal to the Julia sets of any iterate. However, if $f$ is a quasimeromorphic mapping then the iterates $f^n$, $n \in \mathbb{N}$, are not in general quasimeromorphic. Nonetheless, we can prove that a similar property holds for our mappings.

\begin{thm}\label{JuliaIterates}
Let $f: \mathbb{R}^d \to \hat{\mathbb{R}}^d$ be a quasimeromorphic mapping of transcendental type with at least one pole. Then for each $n \in \mathbb{N}$,
\begin{align} \label{JuliaIteratesEquation}
J(f)= \{ &x \in \hat{\mathbb{R}}^d \setminus \overline{\mathcal{O}^{-}_{f}(\infty)} : \card(\hat{\mathbb{R}}^d \setminus \mathcal{O}^{+}_{f^n}(U_{x}))< \infty \text{ for all} \nonumber \\ 
&\text{neighbourhoods } U_{x} \subset \hat{\mathbb{R}}^d \setminus \overline{\mathcal{O}^{-}_{f}(\infty)} \text{ of } x \} \cup \overline{\mathcal{O}^{-}_{f}(\infty)}.
\end{align}
\end{thm}

Quasimeromorphic mappings of transcendental type with at least one pole are closely related to quasiregular mappings in $S$-punctured space (see Section~2.3 below) and to entire quasiregular mappings of transcendental type. For both of these types of mappings, the concept of capacity plays an important role in the Fatou-Julia theory. For the former case, the relationship between the capacity and cardinality of backward orbits has been established; for the latter case, this remains an open conjecture (see \cite{BN1}). Motivated by this, the relationship between cardinality and capacity in the new setting has also been established. Combining this result with \cite[Proposition~3.4]{NS1} gives the following improved theorem.

\begin{thm} \label{backwardOrbitCapacity}
Let $f: \mathbb{R}^d \to \hat{\mathbb{R}}^d$ be a quasimeromorphic mapping of transcendental type with at least one pole. Then $x \in E(f)$ if and only if $cap\left(\overline{\mathcal{O}^{-}_{f}(x)}\right)= 0$.
\end{thm}

Firstly in Section~2 we shall provide some important definitions, alongside some key results and observations regarding quasimeromorphic mappings of transcendental type with at least one pole. Moreover a brief proof of Theorem~\ref{JuliaSetsAgree} will be given. Section~3 will be concerned with the proof of Theorem~\ref{JuliaSetProperties} and Theorem~\ref{JuliaIterates} when $\mathcal{O}^{-}_{f}(\infty)$ is finite, while Section~4 will focus on the case when $\card(\mathcal{O}^{-}_{f}(\infty)) = \infty$. Finally, in Section~5 we shall establish Theorem~\ref{backwardOrbitCapacity}, focusing on the case when $\card(\mathcal{O}^{-}_{f}(\infty)) = \infty$. 

	\section{Preliminary results}	
	
	\subsection{Quasiregular and quasimeromorphic mappings}
	
	We shall briefly recall the definition of quasiregular and quasimeromorphic mappings here, and refer to \cite{Rickman3} for a more detailed introduction.
	
Let $d \geq 2$ and $D \subset \mathbb{R}^d$ be a domain. For $p \in [1,\infty)$, the Sobolev space $W_{p,loc}^{1}(D)$ consists of all functions $f:D \to \mathbb{R}^d$ for which all first order weak partial derivatives exist and are locally in $L^p(D)$. A continuous map $f \in  W_{d,loc}^{1}(D)$ is called quasiregular if there exists some constant $K_{1}\geq 1$ such that 
\begin{equation}\label{qrDefn1}
\left(\sup_{|v|=1}|Df(x)(v)|\right)^d \leq K_{1}J_{f}(x) \text{ a.e.},
\end{equation}
where $Df(x)$ is the derivative of $f(x)$ and $J_{f}(x)$ denotes the Jacobian determinant. 

If $f$ is quasiregular, then there exists some $K_{2} \geq 1$ such that
\begin{equation}\label{qrDefn2}
K_{2}\left(\inf_{|v|=1}|Df(x)(v)|\right)^d \geq J_{f}(x) \text{ a.e.}
\end{equation}
 
The smallest constants $K_{1}$ and $K_{2}$ for which \eqref{qrDefn1} and \eqref{qrDefn2} hold are called the outer and inner dilatation and are denoted $K_{O}(f)$ and $K_{I}(f)$ respectively. If $\max\{K_{O}(f), K_{I}(f)\} \leq K$ for some $K \geq 1$, then we say that $f$ is $K$-quasiregular.
	
Quasiregularity can be extended to mappings between Riemannian manifolds. In particular, consider $\hat{\mathbb{R}}^d = \mathbb{R}^d \cup \{\infty\}$ equipped with the spherical metric. Then for a domain $G \subset \mathbb{R}^d$, a continuous map $f : G \to \hat{\mathbb{R}}^d$ is called quasimeromorphic if every $x \in G$ has a neighbourhood $U_{x}$ such that either $f$ or $M \circ f$ is quasiregular from $U_{x}$ into $\mathbb{R}^d$, where $M:\hat{\mathbb{R}}^d \to \hat{\mathbb{R}}^d$ is a sense-preserving M\"{o}bius map such that $M(\infty) \in \mathbb{R}^d$.

	If $f$ and $g$ are quasiregular mappings, with $f$ defined in the range of $g$, then $f \circ g$ is quasiregular, with 
	\begin{equation*}
	K_{O}(f \circ g) \leq K_{O}(f)K_{O}(g) \text{ and } K_{I}(f \circ g) \leq K_{I}(f)K_{I}(g).
	\end{equation*}
	
Similarly, if $g$ is quasiregular and $f$ is quasimeromorphic in the range of $g$, then $f \circ g$ is quasimeromorphic and the same inequalities as above hold.
	
	Many properties of analytic and meromorphic functions have analogues for quasiregular and quasimeromorphic mappings respectively. For instance, Reshetnyak showed in \cite{Reshetnyak1, Reshetnyak2} that non-constant quasiregular mappings are discrete, open, and sense-preserving. 
	
	Denote the region between two spheres centred at the origin of radii $0 \leq r < s \leq \infty$, by
\begin{equation*}
A(r,s) = \{ x \in \mathbb{R}^d : r < |x| < s \}.
\end{equation*}
An important result is that of Rickman, who proved the following analogue of Picard's theorem in \cite{Rickman1, Rickman2}.

	\begin{thm} \label{BigPicard}
Let $d \geq 2$, $K\geq 1$, $\rho>0$ and let $f : A(\rho, \infty) \to \hat{\mathbb{R}}^d\setminus\{a_1, a_2, \dots, a_p \}$ be a $K$-quasimeromorphic mapping with the $a_{i} \in \hat{\mathbb{R}}^d$ distinct for $i = 1, 2, \dots, p$. Let $q = q(d,K)$ be Rickman's constant. Then $f$ has a limit at $\infty$ whenever $p \geq q$.
\end{thm}	

	As a remark, for $K$-quasiregular mappings the quantity $q -1$ is also referred to as Rickman's constant, to compensate for the fact that infinity is omitted.
	
As we are considering mappings of transcendental type, the above theorem can be used to establish the cases required for further analysis. Indeed, suppose that $f$ is a quasimeromorphic mapping of transcendental type with at least one pole. If $\card(\mathcal{O}^{-}_{f}(\infty))\geq q$, then by Theorem~\ref{BigPicard} there exists some $x \in \mathcal{O}^{-}_{f}(\infty)$ such that $\card(f^{-1}(x)) = \infty$. Therefore we have two cases to consider: either $2 \leq \card(\mathcal{O}^{-}_{f}(\infty)) <q$ or $\card(\mathcal{O}^{-}_{f}(\infty)) = \infty$. 
	
	\subsection{Capacity of a condenser}
Let $A \subset \mathbb{R}^d$ be an open set and let $E \subset A$ be non-empty and compact. The pair $(A,E)$ is called a condenser and the (conformal) capacity of $(A,E)$, denoted $\capacity(A,E)$, is defined by
\begin{equation*}
\capacity(A,E) := \inf_{u}\int_{A}|\nabla u|^d dm,
\end{equation*}
where the infimum is taken over all non-negative functions $u \in C_{0}^{\infty}(A)$ such that for all $x \in E$, $u(x) \geq 1$.

It was shown by Reshetnyak \cite{Reshetnyak1} that if $\capacity(A,E)=0$ for some bounded open set $A \supset E$, then $\capacity(A',E)=0$ for all bounded open sets $A' \supset E$. In this case, we say that $E$ has zero capacity and write $\capacity(E)=0$; otherwise we say that $E$ has positive capacity and write $\capacity(E)>0$. For an unbounded closed set $E \subset \mathbb{R}^d$, we say that $\capacity(E)=0$ if $\capacity(C)=0$ for every compact set $C \subset E$. It is known that sets of zero capacity have Hausdorff dimension zero; see \cite[Theorem 4.1]{Wa2}. Conversely, it is known that finite sets have zero capacity. Thus, informally, sets of zero capacity are `small'.
	
	\subsection{Quasiregular mappings in $S$-punctured space}
To analyse the case when $2 \leq \card(\mathcal{O}^{-}_{f}(\infty)) <q$, we will need to consider the behaviour of quasiregular mappings in punctured space. The iterative theory of such mappings has been studied by Nicks and Sixsmith in \cite{NS1}, thus we shall only state the definition and a few key results here.

Let $d \geq 2, n \in \mathbb{N}$ be fixed and let $S := \{\infty, s_{1}, s_{2}, \dots, s_{n}\}$ be a finite set of distinct points in $\hat{\mathbb{R}}^d$. We note that it is important that $n \geq 1$, so $\card(S) \geq 2$. Now a quasiregular mapping $g: \hat{\mathbb{R}}^d \setminus S \to \hat{\mathbb{R}}^d \setminus S$ is said to be of $S$-transcendental type if $S$ coincides with the set of essential singularities of $g$. The Julia set $J_{S}(g)$ is then defined as
\begin{align}\label{SPunctDef'n}
J_{S}(g) = \{ &x \in \hat{\mathbb{R}}^d \setminus S : \card(\hat{\mathbb{R}}^d \setminus \mathcal{O}^{+}_{g}(U_{x}))< \infty \text{ for all} \nonumber \\ 
&\text{neighbourhoods } U_{x} \subset \hat{\mathbb{R}}^d \setminus S \text{ of } x \}.
\end{align}

We summarise some of the results found in \cite{NS1} in the following theorem. Here, any closure is taken with respect to $\hat{\mathbb{R}}^d \setminus S$, unless stated otherwise.

\begin{thm}[\cite{NS1}]\label{J_S(g)Properties}
Let $S \subset \hat{\mathbb{R}}^d$ be a finite set with $\infty \in S$ and suppose that $g : \hat{\mathbb{R}}^d \setminus S \to \hat{\mathbb{R}}^d \setminus S$ is a quasiregular map of $S$-transcendental type. Then the following hold.
\begin{description}
\item[(a)] $J_{S}(g)$ is infinite and perfect.
\item[(b)] $x \in J_{S}(g)$ if and only if $g(x) \in J_{S}(g)$.
\item[(c)] For all $x \in \mathbb{R}^d \setminus E(g)$, we have $J_{S}(g) \subset \overline{\mathcal{O}^{-}_{g}(x)}$.
\item[(d)] For all $x \in J_{S}(g) \setminus E(g)$, we have $J_{S}(g) = \overline{\mathcal{O}^{-}_{g}(x)}$.
\item[(e)] $J_{S}(g) = J_{S}(g^k)$ for each $k \in \mathbb{N}$.
\item[(f)] The closure of all components of $J_{S}(g)$ with respect to $\hat{\mathbb{R}}^d$ meet $S$.
\end{description}
\end{thm}

\subsection{Proof of Theorem~\ref{JuliaSetsAgree}}
	
Let $f: \mathbb{C} \to \hat{\mathbb{C}}$ be a transcendental meromorphic function with at least one pole, let $\mathcal{F}(f)$ denote the Fatou set for $f$ and let $J_{mero}(f) := \hat{\mathbb{C}} \setminus \mathcal{F}(f)$ denote the classical Julia set of $f$. We shall identify $\mathbb{C}$ with $\mathbb{R}^2$ in the usual way.

If $x \in \hat{\mathbb{C}} \setminus J(f)$, then by \eqref{JuliaSetDefn} there exists some neighbourhood $U_{x}$ of $x$ such that $f^n$ is well defined for all $n \in \mathbb{N}$ and $\card(\hat{\mathbb{C}} \setminus \mathcal{O}_{f}^{+}(U_x)) = \infty$. Now $\{f^n: n \in \mathbb{N}\}$ forms a normal family on $U_{x}$ by Montel's Theorem, hence $x \in \mathcal{F}(f)$. 

Conversely if $x \in \mathcal{F}(f)$, then there exists an open neighbourhood $V_{x}$ of $x$ such that $f^n$ is well defined for all $n \in \mathbb{N}$ and $\{f^n: n \in \mathbb{N}\}$ forms a normal family on $V_{x}$. Since $\mathcal{F}(f)$ is completely invariant, then $\mathcal{O}_{f}^{+}(V_{x}) \subset \mathcal{F}(f)$, therefore $J_{mero}(f) \subset \hat{\mathbb{C}} \setminus \mathcal{O}_{f}^{+}(V_x)$. Finally, as $J_{mero}(f)$ is infinite then $\card(\hat{\mathbb{C}} \setminus \mathcal{O}_{f}^{+}(V_x)) = \infty$. Hence $x \in \hat{\mathbb{C}} \setminus J(f)$ as required.

	\section{Proof of Theorem~\ref{JuliaSetProperties} when $\mathcal{O}^{-}_{f}(\infty)$ is finite}
	
	Suppose that $f: \mathbb{R}^d \to \hat{\mathbb{R}}^d$ is a $K$-quasimeromorphic mapping of transcendental type and $2 \leq \card(\mathcal{O}^{-}_{f}(\infty)) < q$.
	As infinity is an essential singularity for $f$ and $\card(\mathcal{O}^{-}_{f}(\infty)) < q$, then we find that $\mathcal{O}^{-}_{f}(\infty)$ forms the set of essential singularities for $f^{q}$ and for higher iterates of $f$. In addition, $\mathcal{O}^{-}_{f}(\infty)$ is omitted from the image of $f^{q}$ and higher iterates. Therefore, by setting $S := \mathcal{O}^{-}_{f}(\infty)$, it follows that $f^n: \hat{\mathbb{R}}^d \setminus S \to \hat{\mathbb{R}}^d \setminus S$ is a quasiregular mapping of $S$-transcendental type for all $n \geq q$. From this, we may appeal to the Fatou-Julia theory for quasiregular mappings on $S$-punctured space. This provides an approach to proving Theorem~\ref{JuliaSetProperties}.
	
	Firstly, we shall aim to prove a result concerning the relationship between $J(f)$ and $J_{S}(f^n)$ for $n \geq q$. For this, we require a few observations. Indeed, note that by applying Theorem~\ref{J_S(g)Properties}(e) twice, then for all $n \geq q$,
	\begin{equation}\label{(d)Extension}
	J_{S}(f^n) = J_{S}(f^{nq}) = J_{S}(f^{q}).
	\end{equation}
	
Further, by Theorem~\ref{J_S(g)Properties}(b) we can see that, $f(x) \in J_{S}(f^{q})$ if and only if $f^{q+1}(x) \in  J_{S}(f^{q})$, and $f^{q+1}(x) \in J_{S}(f^{q+1})$ if and only if $x \in J_{S}(f^{q+1})$. Since $J_{S}(f^{q}) = J_{S}(f^{q+1})$ by \eqref{(d)Extension}, then we conclude that
\begin{equation}\label{J_S(f)InvarianceViaf}
f(x) \in J_{S}(f^{q}) \text{ if and only if } x \in J_{S}(f^{q}).
\end{equation}

\begin{thm}\label{JuliaRelations}
Let $f: \mathbb{R}^d \to \hat{\mathbb{R}}^d$ be a $K$-quasimeromorphic mapping of transcendental type and at least one pole, such that $S:= \mathcal{O}^{-}_{f}(\infty)$ is finite. Then for all $n \geq q$, 
\begin{equation}
J(f) = J_{S}(f^n) \cup S.
\end{equation}
\end{thm}
	
\begin{proof}
	By \eqref{(d)Extension} and \eqref{JuliaSetDefn}, it will suffice to prove that $J(f) \setminus \mathcal{O}^{-}_{f}(\infty) = J_{S}(f^{q})$. Indeed, firstly note that the reverse inclusion is clear. This is because for any open neighbourhood $U_{x} \subset \hat{\mathbb{R}}^d \setminus \mathcal{O}^{-}_{f}(\infty)$ of a point $x \in J_{S}(f^{q})$, we have
\begin{equation*}
\hat{\mathbb{R}}^d \setminus \mathcal{O}^{+}_{f}(U_{x}) \subset \hat{\mathbb{R}}^d \setminus \mathcal{O}^{+}_{f^{q}}(U_{x}).
\end{equation*}
This means that if $\hat{\mathbb{R}}^d \setminus \mathcal{O}^{+}_{f^{q}}(U_{x})$ is finite, then $\hat{\mathbb{R}}^d \setminus \mathcal{O}^{+}_{f}(U_{x})$ is finite as well.

For the other direction, let $x \in \hat{\mathbb{R}}^d \setminus \mathcal{O}^{-}_{f}(\infty)$ be such that for any open neighbourhood $V_{x} \subset \hat{\mathbb{R}}^d \setminus \mathcal{O}^{-}_{f}(\infty)$ of $x$, then $\card(\hat{\mathbb{R}}^d \setminus \mathcal{O}^{+}_{f}(V_{x}))$ is finite. Since $J_{S}(f^{q})$ is infinite from Theorem~\ref{J_S(g)Properties}(a), then we must have that $\mathcal{O}^{+}_{f}(V_{x}) \cap J_{S}(f^{q}) \neq \varnothing$. It follows by \eqref{J_S(f)InvarianceViaf} that $V_{x} \cap J_{S}(f^{q}) \neq \varnothing$. Finally as $J_{S}(f^{q})$ is closed in $\hat{\mathbb{R}}^d \setminus  \mathcal{O}^{-}_{f}(\infty)$ and $V_{x}$ was an arbitrary open neighbourhood of $x$, then $x \in J_{S}(f^{q})$ as required.
\end{proof}
	
	With Theorem~\ref{JuliaRelations} established, we can turn to the proof of Theorem~\ref{JuliaSetProperties}. Indeed, firstly observe that Theorem~\ref{J_S(g)Properties} states that many properties of the classical Julia set hold for $J_{S}(f^n)$ with $n \geq q$. In particular, Theorem~\ref{J_S(g)Properties}(f) gives us that the closure of every component $Y \subset \hat{\mathbb{R}}^d \setminus \mathcal{O}^{-}_{f}(\infty)$ of $J_{S}(f^{q})$ meets $\mathcal{O}^{-}_{f}(\infty)$. This means with trivial extensions to the arguments in \cite{NS1}, to encompass the case when $x \in \mathcal{O}^{-}_{f}(\infty)$, we find that the properties listed in Theorem~\ref{J_S(g)Properties} also hold for $J_{S}(f^n) \cup S$ for $n \geq q$. Therefore $\textbf{(i)}$-$\textbf{(v)}$ follows immediately from Theorem~\ref{J_S(g)Properties}.
	
	 Finally, $\textbf{(vi)}$ follows from $\textbf{(iv)}$. This is because for any $x \in \mathbb{R}^d \setminus E(f)$, any open set $U \subset \mathbb{R}^d$ intersecting $J(f)$ must also non-trivially intersect $\mathcal{O}^{-}_{f}(x)$.

	\begin{proof}[Proof of Theorem~\ref{JuliaIterates}]
	  Firstly, note that the case when $\card(\mathcal{O}^{-}_{f}(\infty)) = \infty$ is trivial by the remark immediately after Definition~\ref{JuliaSetDefn}.
	  
	 Now set $S = \mathcal{O}^{-}_{f}(\infty)$ and first note that the reverse inclusion in \eqref{JuliaIteratesEquation} is clear. This is because for any set $U \subset \hat{\mathbb{R}}^d \setminus S$ such that $\card(\mathcal{O}_{f^n}^{+}(U_{x}))$ is finite, then $\card(\mathcal{O}_{f}^{+}(U_{x}))$ is finite.
	  
	  For the other direction, note that for any given $n \in \mathbb{N}$ we have $J(f) \setminus S = J_{S}(f^{n{q}})$ by Theorem~\ref{JuliaRelations}. Now let $x \in J(f)\setminus S$ so for any neighbourhood $V_{x} \subset \hat{\mathbb{R}}^d \setminus S$ of $x$, we have that $\card(\hat{\mathbb{R}}^d \setminus\mathcal{O}_{f^{nq}}^{+}(V_{x}))< \infty$. This implies that $\card(\hat{\mathbb{R}}^d \setminus\mathcal{O}_{f^n}^{+}(V_{x}))< \infty$. Since $V_{x}$ was an arbitrary neighbourhood of $x$, then the result follows.
	  \end{proof}

\section{Proof of Theorem~\ref{JuliaSetProperties} when $\mathcal{O}^{-}_{f}(\infty)$ is infinite}

Throughout this section, we shall assume that $f: \mathbb{R}^d \to \hat{\mathbb{R}}^d$ is a quasimeromorphic mapping of transcendental type and $\card(\mathcal{O}^{-}_{f}(\infty)) = \infty$. We recall by an earlier remark that $J(f) = \overline{\mathcal{O}^{-}_{f}(\infty)}$ in this case; we shall use this to prove Theorem~\ref{JuliaSetProperties}.

Firstly observe that by the definition, $\textbf{(i)}$ is clearly satisfied and $J(f)$ is closed. Next, $\textbf{(ii)}$ will follow from the following lemma.

\begin{lem} 
$J(f)$ does not contain any isolated points, hence $J(f)$ is perfect.
\end{lem}

\begin{proof}
To prove this, it shall suffice to show that 
for each $x \in \mathcal{O}^{-}_{f}(\infty)$, every open neighbourhood $U_{x}$ of $x$ is such that $(U_{x} \setminus\{x\}) \cap \mathcal{O}^{-}_{f}(\infty) \neq \varnothing$. Indeed, fix some arbitrary $x \in \mathcal{O}^{-}_{f}(\infty)$, so there exists some $N \geq 0$ such that $f^N(x)=\infty$. Let $U_{x}$ be an open neighbourhood of $x$. Then there exists some open neighbourhood $V_{x} \subset U_{x}$ such that $f^N$ is quasimeromorphic on $V_{x}$ and
\begin{equation*}
\overline{V_{x}} \cap \bigcup_{n=0}^{N}f^{-n}(\infty) = \{x\},
\end{equation*}
so $f^N(V_{x})$ is an open set around infinity. 

Since $\card(\mathcal{O}^{-}_{f}(\infty)) = \infty$, then by Theorem~\ref{BigPicard} there exists some $s \in (\mathcal{O}^{-}_{f}(\infty) \setminus \{\infty\}) \cap f^{N}(V_{x})$. This implies that there exists some $v_{s} \in V_{x}$ such that $f^N(v_{s}) = s$, whence $v_{s} \in \mathcal{O}^{-}_{f}(\infty)$. It remains to note that $v_{s} \neq x$ since $f^{N}(v_{s}) \neq \infty$. Thus $(V_{x} \setminus \{x\}) \cap \mathcal{O}^{-}_{f}(\infty) \neq \varnothing$, and so $(U_{x} \setminus \{x\}) \cap \mathcal{O}^{-}_{f}(\infty) \neq \varnothing$ as required.
\end{proof}

For $\textbf{(iii)}$, first let $x \in J(f) \setminus \{\infty\}$. If $x \in \mathcal{O}^{-}_{f}(\infty) \setminus \{\infty\}$, then there exists some $N \in \mathbb{N}$ such that $f^N(x) = \infty$. Now $f(x)$ is defined with $f^{N-1}(f(x)) = \infty$, thus $f(x) \in J(f)$. If $x \in \overline{\mathcal{O}^{-}_{f}(\infty)} \setminus \mathcal{O}^{-}_{f}(\infty)$, then there exists $x_{n} \in \mathcal{O}^{-}_{f}(\infty) \setminus \{\infty\}$ such that $x_{n} \to x$ as $n \to \infty$. It then follows that $f(x_{n})$ exists for each $n \in \mathbb{N}$. As $f$ is continuous, then $f(x_{n}) \to f(x)$ as $n \to \infty$. Therefore as $J(f)$ is closed, we conclude that $f(x) \in J(f)$. In particular since $x \not \in \mathcal{O}^{-}_{f}(\infty)$, then $f(x)\not \in \mathcal{O}^{-}_{f}(\infty)$ either, so $f(x) \in J(f) \setminus \mathcal{O}^{-}_{f}(\infty)$.

For the other direction, let $f(x) \in J(f)$ for some $x \in \mathbb{R}^d$. If $f(x) \in \mathcal{O}^{-}_{f}(\infty)$, then $x \in \mathcal{O}^{-}_{f}(\infty) \setminus \{\infty\}$ by definition. So suppose that $f(x) \in \overline{\mathcal{O}^{-}_{f}(\infty)} \setminus \mathcal{O}^{-}_{f}(\infty)$ and let $U$ be a neighbourhood of $x$. Since $f(x)$ is defined and $f$ is an open mapping, then there exists some neighbourhood $V$ of $f(x)$ such that $f(U) \supset V$. As $f(x)$ is a limit point of $\mathcal{O}^{-}_{f}(\infty)$, there exists some $y_{n} \to f(x)$ such that $y_{n} \in V \cap \mathcal{O}^{-}_{f}(\infty)$ for all large $n$. This means for all large $n \in \mathbb{N}$ there exists $x_{n} \in U \cap \mathcal{O}^{-}_{f}(\infty)$ with $f(x_{n})=y_{n}$. As $U$ was an arbitrary neighbourhood, then we must have $x_{n} \to x$ as $n \to \infty$, thus $x \in \overline{\mathcal{O}^{-}_{f}(\infty)} \setminus \{\infty\} = J(f) \setminus \{\infty\}$.

Next, we shall prove $\textbf{(iv)}$ in the following lemma.

\begin{lem}\label{(iv)}
For all $x \in \hat{\mathbb{R}}^d \setminus E(f)$, $J(f) \subset \overline{\mathcal{O}^{-}_{f}(x)}$.
\end{lem}

\begin{proof}
Let $x \in \hat{\mathbb{R}}^d \setminus E(f)$. Note that the result would be trivial if $x = \infty$, so assume that $x \neq \infty$. Now since $x \not \in E(f)$, then $\card(\mathcal{O}^{-}_{f}(x))=\infty$. As $E(f)$ is a finite set, then by Theorem~\ref{BigPicard} there exists $w_{n} \to \infty$ such that for each $n \in \mathbb{N}$, $w_{n} \in \mathcal{O}^{-}_{f}(x)$ and $\card(\mathcal{O}^{-}_{f}(w_{n}))=\infty$.

Now let $y \in \mathcal{O}^{-}_{f}(\infty)$ and let $\varepsilon>0$ be given. Then there exists some $k \in \mathbb{N}_{0}$ such that $f^k(y)=\infty$. Since $f^k$ is discrete, then there exists some $R >0$ and $0< \delta \leq \varepsilon$ such that $f^k$ is quasimeromorphic on $B(y, \delta)$ and $f^k(B(y,\delta)) \supset A(R, \infty)$. As $w_{n} \to \infty = f^k(y)$, then there exists some $N \in \mathbb{N}$ such that $w_{n} \in A(R, \infty)$ whenever $n \geq N$. Hence for all $n \geq N$, there exists $y_{n} \in B(y, \delta)$ such that $y_{n} \in f^{-k}(w_{n}) \subset \mathcal{O}^{-}_{f}(x)$. Since $\varepsilon>0$ was arbitrary, then $y_{n} \to y$ as $n \to \infty$, therefore $y \in \overline{\mathcal{O}^{-}_{f}(x)}$ as required.
\end{proof}

It is easy to see that $\textbf{(v)}$ follows immediately from $\textbf{(iv)}$; $J(f) \subset \overline{\mathcal{O}^{-}_{f}(x)}$ follows from Lemma~\ref{(iv)}, while the other direction follows from $\textbf{(iii)}$ and the fact that $J(f)$ is closed. Further, $\textbf{(vi)}$ also follows immediately from $\textbf{(iv)}$. This is because for every $x \in \hat{\mathbb{R}}^d \setminus E(f)$, any open set $U \subset \hat{\mathbb{R}}^d$ intersecting $J(f)$ must also non-trivially intersect $\mathcal{O}^{-}_{f}(x)$. This completes the proof of Theorem~\ref{JuliaSetProperties}.

\section{Proof of Theorem~\ref{backwardOrbitCapacity}}

	In this section we shall establish Theorem~\ref{backwardOrbitCapacity}, which is motivated by the conjecture in \cite{BN1}. This conjecture is  concerned with whether the capacity definition used in \cite{BN1} to define the Julia set for mappings without poles is equivalent to a cardinality condition as used in Definition~\ref{JuliaSetDefn}.

	Suppose that $f: \mathbb{R}^d \to \hat{\mathbb{R}}^d$ is a quasimeromorphic mapping of transcendental type with at least one pole. By considering a large iterate of $f$ and using Proposition~3.4 in \cite{NS1}, the case of Theorem~\ref{backwardOrbitCapacity} when $\mathcal{O}^{-}_{f}(\infty)$ is finite has already been established. Hence within this section, we shall only consider the case when $\card(\mathcal{O}^{-}_{f}(\infty))= \infty$.
	
	 We shall first prove a covering result for the neighbourhoods of poles. To simplify notation, we shall say that $x$ is an $m$-prepole of $f$ if $f^{m}(x)=\infty$ for some $m \in \mathbb{N}$, where a 1-prepole is precisely the same as a pole.  Now observe that if $\card(\mathcal{O}^{-}_{f}(\infty)) = \infty$, then there exist some point $y \in \mathcal{O}^{-}_{f}(\infty)$ and $N \in \mathbb{N}$ such that $y$ is an $N$-prepole of $f$ and $\card(f^{-1}(y)) = \infty$.

\begin{lem} \label{EExists}
	Let $f: \mathbb{R}^d \to \hat{\mathbb{R}}^d$ be a quasimeromorphic mapping of transcendental type. Suppose that there exists an open bounded neighbourhood $U \subset \mathbb{R}^d$ of an $N$-prepole of $f$ such that $f^N$ is quasimeromorphic on $U$ and $f^{-1}(u)$ is infinite for all $u \in \overline{U}$. Then given any $r >0$, there exists an open bounded set $E_{U} \subset A(r,\infty)$ such that $f^N(U) \supset \overline{E_{U}}$ and $f(E_{U}) \supset \overline{U}$.
\end{lem}
\begin{proof}
	Since $f^N$ is quasimeromorphic on $U$ and $U$ is an open set containing an $N$-prepole of $f$, then $f^N(U)$ is an open set covering infinity. Now by assuming without loss of generality that $r>0$ is sufficiently large, then $\overline{A(r, \infty)} \subset f^N(U)$. 
	
	Let $u \in \overline{U}$. As $f^{-1}(u)$ is infinite and $f$ is a discrete mapping, there exists $x_u \in A(r+1, \infty)$ such that $f(x_u) = u$. As $f$ is open, then $B_{u}:= B^d(x_u, 1) \subset A(r, \infty)$ and $f(B_{u}) \supset B^d(u, \delta_{u})$ for some $\delta_{u}>0$. Thus we can construct an open cover 
	\begin{equation*}	
	\bigcup_{u \in \overline{U}}f(B_{u}) \supset \bigcup_{u \in \overline{U}}B^d(u, \delta_{u}) \supset \overline{U}. 
	\end{equation*}
	
	As $\overline{U}$ is non-empty and bounded, then it is compact. Hence there exists some $n \in \mathbb{N}$ and $u_{i} \in \overline{U}$, $i = 1,2,\dots ,n$, such that 
	\begin{equation} \label{property1}
	\bigcup_{i=1}^{n}f(B_{u_{i}}) \supset \overline{U}.
	\end{equation}
	
	Now define 
	\begin{equation} \label{property3}
	E_{U} := \bigcup_{i=1}^{n}B_{u_{i}}. 
	\end{equation}
	
	Observe that $B_{u_{i}} \subset A(r, \infty)$ for each $i=1,2,\dots,n$, so $\overline{E_{U}}~\subset~\overline{A(r, \infty)}~\subset f^N(U)$. Further, as each $B_{u_{i}}$ is bounded then $E_{U}$ is bounded. Hence from \eqref{property1} and \eqref{property3}, $E_{U}$ is the bounded set as required. 
\end{proof}

\begin{proof}[Proof of Theorem~\ref{backwardOrbitCapacity}]

Firstly, we shall show that for some $N \in \mathbb{N}$, there exists an $N$-prepole $y$ of $f$ and a bounded neighbourhood $\overline{U_{y}}$ of $y$ such that for all $w \in \overline{U_{y}}$, $\capacity\left(\overline{\mathcal{O}^{-}_{f}(w)}\right)> 0$. Indeed, as $\card(\mathcal{O}^{-}_{f}(\infty))=\infty$, then by Theorem~\ref{BigPicard} and the definition of the backwards orbit there must exist some $N \in \mathbb{N}$, $y \in \mathcal{O}^{-}_{f}(\infty)$ and some bounded neighbourhood $U_{y}$ of $y$ such that $y$ is an $N$-prepole of $f$, $\card(f^{-1}(u)) = \infty$ for all $u \in \overline{U_{y}}$ and $f^N$ is quasimeromorphic on $U_{y}$.

Now by repeatedly applying Lemma~\ref{EExists}, there exists a collection of open bounded sets $\{E_{n}: n \in \mathbb{N}\}$ with pairwise disjoint closures such that for all $n \in \mathbb{N}$, 
\begin{equation*}
f^N(U_{y}) \supset \overline{E_{n}} \text{ and } f(E_{n}) \supset \overline{U_{y}}.
\end{equation*}

Since these sets have pairwise disjoint closures, then there must exist pairwise disjoint closed sets $V_{n} \subset U_{y}$ such that for each $n \in \mathbb{N}$, $f^N(V_{n}) \supseteq \overline{E_{n}}$. In particular, this means that $f^{N+1}(V_{n}) \supset \overline{U_{y}}$.

By applying Lemma~2.9, Lemma~2.10 and Lemma~2.11 from \cite{BN1} with $m >K_{I}(f^{N+1})$ closed sets $V_{1}, V_{2}, \dots, V_{m}$, then for all $w \in \overline{U_{y}}$, we have that $\capacity\left(\overline{\mathcal{O}^{-}_{f^{N+1}}(w)}\right)> 0$. Since $\overline{\mathcal{O}^{-}_{f^{N+1}}(w)} \subset \overline{\mathcal{O}^{-}_{f}(w)}$ by definition, then for all $w \in \overline{U_{y}}$, $\capacity(\overline{\mathcal{O}^{-}_{f}(w)})>0$.

Now take some $x \in \mathbb{R}^d \setminus E(f)$, so $\card(\mathcal{O}^{-}_{f}(x)) = \infty$. Then by definition and Theorem~\ref{BigPicard}, there exists some $\alpha \in \mathcal{O}^{-}_{f}(x) \cap f^{N}(U_{y})$, where $y$ and $U_{y}$ are as above. In particular, this means that there exists some $u_{\alpha} \in U_{y}$ such that $f^N(u_{\alpha})=\alpha$. Hence $u_{\alpha} \in \mathcal{O}^{-}_{f}(x)$. As $u_{\alpha} \in \overline{U_{y}}$, then $\capacity\left(\overline{\mathcal{O}^{-}_{f}(u_{\alpha})}\right)> 0$. Therefore we have that $\capacity\left(\overline{\mathcal{O}^{-}_{f}(x)}\right)> 0$ as required.
\end{proof}

\end{document}